\documentclass[a4paper, oneside, reqno,12pt]{amsart}

\usepackage[left = 2.5 cm, right = 2.5 cm,
]{geometry}

\usepackage{
amssymb,
amsfonts,
amsmath,
mathtools,
amscd,
amsthm,
commath,
esdiff,
cases
}

\usepackage{stmaryrd}
\usepackage{textcomp}

\usepackage[pagewise]{lineno}
\usepackage[utf8]{inputenc}	
\usepackage{hyperref, color}
\hypersetup{
	colorlinks = true,
	linkcolor = red,
	filecolor = blue,
	urlcolor = blue,
	citecolor = blue,
    pdftitle = {Domingos-Onnis}
}

\usepackage{url}
\usepackage{epsfig,psfrag}
\usepackage{graphicx,subfigure}

\usepackage{tikz}
\usetikzlibrary{decorations.pathreplacing}
\usetikzlibrary{patterns}
\usepackage{pstricks}
\usepackage{pst-plot}
\usepackage{pst-node}
\usepackage{pgfplots}
\pgfplotsset{compat=1.18}

\newcommand{\R}{\mathbb R}
\newcommand{\C}{\mathbb C}

\newcommand{\Q}{\mathbb Q}
\newcommand{\Z}{\mathbb Z}
\newcommand{\N}{\mathbb N}

\renewcommand{\S}{\mathbb S}

\usepackage{marginnote}

\newtheorem{theorem}{Theorem}[section]
\newtheorem{lemma}[theorem]{Lemma}
\newtheorem{proposition}[theorem]{Proposition}
\newtheorem{corollary}[theorem]{Corollary}

\newtheorem*{theorem*}{Theorem}

\newtheorem*{lemma*}{Lemma}

\theoremstyle{definition}

\newtheorem{example}[theorem]{Example}
\newtheorem{remark}[theorem]{Remark}

\numberwithin{equation}{section}
\numberwithin{figure}{section}

\begin{document}

\title[Spherical Ricci tori
 with rotational symmetry]{Spherical Ricci tori
 with rotational symmetry}

\author[I. Domingos]{Iury Domingos}
\address{Universidade Federal de Alagoas\\
Av. Manoel Severino Barbosa S/N,
57309-005 Arapiraca - AL, Brazil}
\email{iury.domingos@arapiraca.ufal.br}

\author[I. I. Onnis]{Irene I. Onnis}
\address{Universit\`a degli Studi di Cagliari\\
Dipartimento di Matematica e Informatica\\
Via Ospedale 72, 09124 Cagliari, Italy}
\email{irenei.onnis@unica.it}

\thanks{The authors were partially supported by the Brazilian National Council for Scientific and Technological Development (CNPq), grant no.~409513/2023-7, and by the ICTP-INDAM Research in Pairs Programme 2025. I.I.~Onnis was supported by the Fondazione di Sardegna and partially funded by the PNRR e.INS Ecosystem of Innovation for Next Generation Sardinia (CUP~F53C22000430001, codice MUR ECS00000038).
}

\keywords{Ricci surfaces, Minimal surfaces, Rotational surfaces.}

\subjclass{53C42,  53C40}

\begin{abstract}
In this article, we study $c$-spherical Ricci metrics, that is, Riemannian metrics whose Gaussian curvature $K$ satisfies
\begin{equation*}
    (K - c)\Delta K - |\nabla K|^2 - 4K(K - c)^2 = 0,
\end{equation*}
for some $c>0$. We explicitly construct a two-parameter family of such metrics with rotational symmetry and show that infinitely many non-isometric examples can be realized on the same torus. Moreover, we investigate their realization as induced metrics on compact rotational surfaces in $\mathbb{S}^3$, establishing the existence of embedded compact spherical Ricci surfaces by controlling a period function associated with the isometric immersion.

\end{abstract}

\maketitle

\section{Introduction}

The study of whether a local minimal isometric immersion of a Riemannian surface \((\Sigma,\dif\sigma^2)\) into a $3$-dimensional space form \(\mathbb{M}_c^3\) of constant sectional curvature \(c\) exists can be approached intrinsically. A natural obstruction to the existence of such a minimal immersion arises from the Gauss equation: namely, the Gaussian curvature \(K\) of \(\dif\sigma^2\) must satisfy \(K \leq c\).

In the case of \(\mathbb{R}^3\), G. Ricci-Curbastro \cite{Ricci-Curbastro} showed that if \(K < 0\) and $\Delta\log(-K) = 4K$, a local minimal immersion exists. This was generalized by B.~H.~Lawson \cite{Lawson-1970}, who proved that a surface with \(K < c\) admits a local minimal immersion into \(\mathbb{M}_c^3\) if  \(K\) satisfies $\Delta\log(c-K) = 4K$, or equivalently,
\begin{equation}\label{Ricci-condition}
(K - c)\Delta K - |\nabla K|^2 - 4K(K - c)^2 = 0.
\end{equation}
When \(c = 0\), this condition is often referred to as the \emph{Ricci condition}.

In this manuscript, we restrict our attention to the case \(c > 0\). Following Lawson \cite{Lawson-71}, a Riemannian surface \((\Sigma,\dif\sigma^2)\) satisfying this equation is called a \emph{$c$-spherical Ricci surface}, with \(\dif\sigma^2\) a \emph{$c$-spherical Ricci metric}. Clearly, the condition depends only on the metric and does not require \((K - c)\) to be negative.

Our interest lies in constructing $c$-spherical Ricci tori with rotational symmetry. Except for the round metric, no $c$-spherical Ricci metric has \((K-c) \neq 0\) at any point. In the genus-one case, the only known examples were constructed in \cite{Daniel-Zang} and their metrics coincide with those induced on the Delaunay constant mean curvature surfaces in \(\mathbb{R}^3\).

Our construction is based on positive solutions of an ordinary differential equation of the form
\[
(f')^2 + A\,f^2 + B\,f^{-2} = C,
\]
with constants \(A, B, C \in \mathbb{R}\) chosen appropriately. This equation also arises in the study of rotational constant mean curvature (CMC) surfaces in \(\mathbb{S}^3\) \cite{Otsuki1970,Wei2006}, where there is a one-to-one correspondence between positive solutions and rotational CMC surfaces.

Although $c$-spherical Ricci surfaces always admit a local minimal isometric immersion into the $3$-sphere \(\S^3_c\) of sectional curvature $c$, we aim to study complete immersions without extrinsic constraints, constructing explicit compact examples in \(\S^3_c\).

This work is organized as follows: In Section 2, we fix notations and recall previous results on generalized Ricci surfaces. We then focus on generalized Ricci warped metrics of type $(a,0,c)$, expressing the generalized Ricci condition in terms of the warping function, which yields a third-order differential equation.

In Section 3, we reduce the study of the warping function to an associated autonomous system and seek a positive periodic solution, which allows us to construct a two-parameter family of generalized Ricci metrics with rotational symmetry on surfaces of genus one. This intrinsic construction recovers all previously known examples of generalized Ricci tori of type $(a,0,c)$.

In Section 4, we consider $c$-spherical Ricci surfaces (type $(4,0,c)$), explicitly constructing a two-parameter family of rotational metrics. Theorem~\ref{thm:ricci-tori} shows that infinitely many non-isometric metrics in this family can be realized on the same torus.

Finally, in Section 5, we study conditions under which these metrics can be realized as induced metrics on compact rotational surfaces in \(\mathbb{S}^3\) (Theorems~\ref{thm:immersions} and \ref{thm:embedded}), analyzing a period function associated with the isometric immersion and controlling the admissible parameters to guarantee embedded examples.

\subsection*{Acknowledgements} The authors are grateful to Oscar Perdomo for his valuable suggestions and for drawing our attention to the embedded examples discussed in Section 5. I.~Domingos also wishes to express his sincere gratitude to the Mathematics Department of the Università degli Studi di Cagliari, where this work was initiated, and to the Abdus Salam International Centre for Theoretical Physics -- ICTP for its warm hospitality during the Research in Pairs Programme in June 2025, where this work was completed.

\section{Preliminaries}\label{sec:preliminaries}

\subsection{Generalized Ricci metrics with rotational symmetry}\label{sec:generalized-ricci-surfaces}
Given a smooth connected Riemannian surface $(\Sigma, \dif \sigma^2)$ and a point $(a, b, c) \in \R^3$, such as defined in \cite{Daniel-Zang}, the metric $\dif\sigma^2$ is said to be a {\it generalized Ricci metric of type $(a, b, c)$} if its Gaussian curvature $K$ satisfies the equation:
\begin{equation}\label{GRC}
    (K - c) \Delta K - \|\nabla K\|^2 - (a\,K + b)(K - c)^2 = 0,
\end{equation}
where $\nabla$ and $\Delta$ stand the gradient and
the Laplace-Beltrami operators of $\dif\sigma^2$, respectively.
In this case, $(\Sigma, \dif\sigma^2)$ is called a {\it generalized Ricci surface of type $(a, b, c)$}.

In particular, generalized Ricci surfaces of type $(4, 0, 0)$ are Ricci surfaces in the sense of \cite{Moroianu15}.  Therefore, equation \eqref{GRC} will be referred to as the {\it generalized Ricci condition of type $(a, b, c)$}.

We need to recall some results and important facts about generalized Ricci surfaces that will be used throughout this work; for a more detailed exposition, we refer to \cite{Daniel-Zang}. Firstly, observe that on an open set where $(K - c)$ does not vanish,
\[
    \Delta \log |K - c| = \frac{(K - c) \Delta K - |\nabla K|^2}{(K - c)^2}
\]
and, therefore, equation \eqref{GRC} is equivalent to
\begin{equation}\label{GRC-2}
    \Delta \log |K - c| = a\,K + b.
\end{equation}
 Furthermore, in a generalized Ricci surface, either $K = c$ or the zeros of $(K - c)$ are isolated. In particular, the function $(K - c)$ does not change sign.

One can see directly from \eqref{GRC} that given a generalized Ricci metric $\dif\sigma^2$ of type $(a, b, c)$ and $\eta > 0$, then $\eta \dif\sigma^2$ is a generalized Ricci metric of type $(a, b / \eta^2, c / \eta^2)$. Thus, in the case $b = 0$, up to scaling, only the sign of $c$ is relevant in the study of these metrics. Also, metrics with constant Gaussian curvature $K$ are generalized Ricci metrics of type $(a, b, c)$ if and only if $K = c$ or $a\,K + b = 0$.

Regarding the topology, the presence of a generalized Ricci condition may impose obstructions on the existence of compact orientable surfaces of non-constant Gaussian curvature endowed with such a condition. For example, restricting ourselves to the case $b = 0$, the following results hold:
\begin{itemize}
\item For compact generalized Ricci spheres of type $(a, 0, c)$:
\begin{itemize}
    \item if $c = 0$, then $a \in -2 \N^*$ is a necessary and sufficient condition;
    \item if $c \neq 0$, then $a \in -\N^*$ is a necessary condition and $a \in -2 \N^*$ is a sufficient condition.
\end{itemize}
\item For  surfaces of genus $1$, a necessary and sufficient condition is $a\, c > 0$.
\item For compact surfaces with genus $g \geq 2$,
\begin{itemize}
    \item if $c \leq 0$, then $(g - 1)\,a \in \N^*$ is a necessary and sufficient condition;
    \item if $c > 0$, then $(g - 1)\,a \in \N^*$ is a necessary condition.
\end{itemize}
\end{itemize}
In particular, in the class of generalized Ricci surfaces of type $(4, 0, c)$ with non-constant Gaussian curvature, there are no spheres, and tori appear only when $c > 0$; however, examples of compact surfaces with $g \geq 2$ appear in abundance.

\subsection{SO(2)-invariant metrics}
From now on, we turn our attention to the case of generalized Ricci metrics of type $(a, 0, c)$ that possess rotational symmetry.

In what follows, we denote by $\S^1(r)$ the circle in $\R^2$ with radius $r$. For convenience, we set $\S^1 = \S^1(1)$. Note that $\S^1(r) \cong \R / 2\pi r \Z$ for $r > 0$, where $\Z$ denotes the set of all integers.

Let $I \subset \R$ be an open interval and consider  the warped metric $\dif\sigma^2 = \dif s^2 + f(s)^2 \dif t^2$, where $f: I \to \R$ is a positive smooth function and $t\in\S^1$. We can define an action of the special orthogonal group $\mathbf{SO}(2)$ on the surface  $I\times\S^1$ whose orbits are precisely $\{s\} \times \S^1$, for $s \in I$. This is an isometric action, and we say that $\dif\sigma^2$ is a {\it rotationally symmetric metric}.

We restrict ourselves to any open set where $(K - c)$ does not vanish. We observe that the generalized Ricci condition takes an approachable form in terms of the function $f$. By a direct computation using $\dif\sigma^2$ and the expression of its Gaussian curvature $K = -f'' / f$, equation \eqref{GRC-2} yields
\[
    \big[ f\,(\log|K - c|)'  \big]' + a\, f'' = 0,
\]
where $'$ denotes the derivative with respect to $s$. By integrating the above equation and using $f^{a}(K-c) = -f^{a-1}(f'' + cf)$, the problem of finding a generalized Ricci metric $\dif\sigma^2 = \dif s^2 + f(s)^2 \dif t^2$ of type $(a, 0, c)$ is equivalent to find a positive solution $f: I \to \R_+$ of
\begin{equation}\label{ODE}
    \big[f^{a-1}(f'' + c f)\big]' - A\, f^{a-2}(f'' + c f) = 0, \quad \text{for} \quad A \in \mathbb{R}.
\end{equation}

\begin{remark}
If $f: I \to \R$ is a positive smooth function that satisfies $f'' + c\,f = 0$ on $I$, then it is a particular solution of \eqref{ODE}. In this case, the metric $\dif\sigma^2$ has constant Gaussian curvature $c$.
\end{remark}

\begin{remark}
For $(a,b,c) = (4,0,0)$, the rotationally invariant metrics satisfying the Ricci condition, in the sense of \cite{Moroianu15}, were studied and classified by the first author in \cite{dCDS,DSV2023}. In this case, the above equation is equivalent to the first-order ordinary differential equation $$f'(s) f(s) = A\, f(s) + B\, s + C,$$ for appropriate $A, B, C \in \R$.
\end{remark}

\subsection{An associated autonomous system}
By considering $A=0$, equation \eqref{ODE} gives rise to an autonomous second-order differential equation
\begin{equation}\label{f''}
  f'' = m\,f^{1-a} - c\,f,\qquad  m \in \R,
\end{equation}
whose solution produces a metric of constant Gaussian curvature when $ m=0$. To address the general case, for a fixed $ m\in\R^*$ the problem is reduced to analyzing the nonlinear system
\[\begin{cases}
  x' = y, & \\
  y' =  m\, x^{1-a} - c\, x. &
\end{cases}\tag{$S_{m}$}\label{system}\]
The equilibria of \eqref{system} are the points $(x_*, 0)$ such that $m\, x_*^{1-a} - c\, x_* = 0$, corresponding to the constant solutions $x(s) = x_*$ and $y(s) = 0$, provided $x_*$ is well-defined. We may verify that system \eqref{system} is conservative, and so the total energy $E_m(x, y) = y^2/2+P_ m(x)$ is conserved along its solutions, where
\begin{equation}\label{potential}
P_ m(x) = \begin{cases}
  \displaystyle
   -  m \log x + \frac{c\, x^2}{2}, & \mbox{when } a = 2, \\[1em]
  \displaystyle
   - \frac{ m\, x^{2-a}}{2-a} + \frac{c\, x^2}{2}, & \mbox{when } a \neq 2,
\end{cases}
\end{equation}
is the potential function of \eqref{system}. Thus, any solution $(x(s), y(s))$  satisfies $E_{m}(x(s), y(s)) = \ell$, for some $\ell \in \R$. In addition, given an initial condition $(x(0), y(0)) = (x_0, y_0)$, a nonempty curve $\{(x,y):E_{m}(x, y) = \ell\}$ singles out a unique solution of the system, where $\ell$ is such that $(x_0, y_0) \in \{(x,y):E_ m(x, y) = \ell\}$.

\section{Constructing generalized Ricci tori}\label{sec:generalized-ricci}
Before addressing the construction of compact generalized Ricci tori of type \((a, 0, c)\) with rotational symmetry, let us first explain why we focus directly on the genus \(1\) case.

Consider the genus \(0\) case, where we assume \( a \in -\mathbb{N}^* \). Let \( (\Sigma, \dif \sigma^2) \) be a Riemannian $2$-sphere viewed as a warped product \( (0, \delta) \times \mathbb{S}^1 \), endowed with the metric \( \dif s^2 + f(s)^2 \dif t^2 \). Here, \( f(s) \) is a positive function that extends smoothly to \([0, \delta]\). This means that \( f > 0 \) on \( (0, \delta) \) and satisfies
\begin{equation}\label{sphere-condition}
  f(0) = f(\delta) = 0, \quad f'(0) = -f'(\delta) = 1, \quad \text{and} \quad f^{(2n)}(0) = f^{(2n)}(\delta) = 0,
\end{equation}
for all \( n \in \mathbb{N}^* \). These conditions ensure the smoothness of the metric \( \dif s^2 + f(s)^2 \dif t^2 \) on \([0, \delta] \times \mathbb{S}^1\) (see, e.g., \cite{petersen-book}).

To obtain a Ricci metric of type \( (a,0,c) \) on \( \Sigma \), we seek a smooth nonnegative function \( f: [0, \delta] \to \mathbb{R} \) such that \( f(s) \) is a solution of equation \eqref{f''} on \( (0,\delta) \) and satisfies the conditions in \eqref{sphere-condition}, where \( \delta = \delta_{a,c,m} > 0 \) is to be determined. Since \( m \in \mathbb{R}^* \), from \eqref{sphere-condition} it follows that \( -a \in 2\mathbb{N} \) is a necessary condition. Furthermore, for \( -a \in 2\mathbb{N} \), solutions exist and, up to a change of coordinates, they coincide with those in \cite[Proposition 5.7, with \( \tau=0 \)]{Daniel-Zang} for \( c = 0 \), as well as in \cite[Proposition 5.9]{Daniel-Zang} for \( c \in \mathbb{R}^* \). Hence, we proceed directly to the genus \( 1 \) case.

\subsection{Generalized Ricci tori of type \texorpdfstring{$(a,0,c)$}{(a,0,c)} with \texorpdfstring{$a\,c>0$}{ac>0}}
In order to construct examples of generalized Ricci tori of type \((a, 0, c)\) with rotational symmetry, we will restrict our analysis to the case \( a \,c > 0 \). Accordingly, we consider the universal covering by setting \( t \in \mathbb{R} \) and seek a positive \( T \)-periodic solution \( f: I \to \mathbb{R} \) of \eqref{f''}. Once such a solution is obtained, the metric \( \dif s^2 + f(s)^2 \dif t^2 \) defines a Ricci metric of type \((a, 0, c)\) on \( \mathbb{R}^2 \). Additionally, since \( f \) is a \( T \)-periodic function, for any lattice
\[
\Gamma_T(\gamma_1, \gamma_2) = \left(T, \gamma_1\right)\mathbb{Z} \oplus (0 , \gamma_2)\mathbb{Z},
\]
with \( (\gamma_1, \gamma_2) \in \mathbb{R} \times \mathbb{R}^* \), this solution induces a well-defined Ricci metric of type \((a, 0, c)\) on the torus \( \mathbb{R}^2 / \Gamma_T(\gamma_1, \gamma_2) \), by passing the metric \( \dif s^2 + f(s)^2 \dif t^2 \) from \( \mathbb{R}^2 \) onto the quotient space.

We analyze the system \eqref{system} under the assumption that $m \in \R^*$ satisfies $c \, m>0$. Firstly, this condition ensures the existence of a unique equilibrium point of \eqref{system}, corresponding to the constant solution $(x(s),y(s))=(x_*, 0)$, with $x_*=(m/c)^{1/a}>0$.

Since our focus lies on periodic solutions of \eqref{system}, which are implicitly defined by $E_m(x, y) = \ell$, provided that the level set $\{(x,y):E_m(x, y) = \ell\}$ is nonempty for some $\ell \in \R$, we establish below the condition for this set to be a nonempty level set that describes a compact curve surrounding the equilibrium point $(x_*, 0)$.

\begin{proposition}\label{prop:l}
Let $a, c, m \in \R^*$ have the same sign. Then the level set $\{(x,y):E_m(x, y) = \ell\}$ is a compact smooth curve surrounding the equilibrium point $(x_*, 0)$ if and only if
\[P_m(x_*)<\ell<\lim_{x\to 0}P_m(x),\]
where $P_m(x)$ is the potential function of \eqref{system}.
\end{proposition}
\begin{proof}
Let $P_m(x)$ be the potential function defined in \eqref{potential}, for $x \in (0, +\infty)$. Since, by hypothesis, $a, c, m \in \R^*$ have the same sign, it can be verified that $P_m(x)$ attains its minimum at $x_* = (m/c)^{1/a}$, decreases on $(0, x_*)$, and increases on $(x_*, +\infty)$. Moreover:
\begin{enumerate}
    \item If $a \in (-\infty, 0) \cup (0, 2)$, then $P_m(x_*) < 0$, $P_m(x)$ approaches $0$ as $x \to 0$, and diverges to $+\infty$ as $x \to +\infty$.
    \item If $a \geq 2$, then $P_m(x)$ diverges to $+\infty$ as $x \to 0$ or $x \to +\infty$.
\end{enumerate}
Therefore, the level set $\{(x,y) : E_m(x, y) = \ell\}$ defines a compact smooth curve surrounding the equilibrium point $(x_*, 0)$ if and only if the stated condition is satisfied.
\end{proof}

From Proposition \ref{prop:l} it follows that system~\eqref{system} admits a one-parameter family of periodic solutions within $(0, +\infty) \times \R$. From this family, we derive a two-parameter family of generalized Ricci metrics of type $(a, 0, c)$ with rotational symmetry. Specifically, let $a, c \in \R^*$ with $a \,c > 0$, and $m, \ell\in\R$ such that $c\,m>0$ and
\[\mathrm{min} \, P_m(x)<\ell<\lim_{x\to 0}P_m(x).\]
We let $(f_{m,\ell}(s), f_{m,\ell}'(s))$ denote the periodic solution contained in the level set $\{(x,y) : E_m(x, y) = \ell\}$, with minimal period $T_{m,\ell} > 0$. By construction, $f_{m,\ell}(s)$ is a positive global solution of the equation $f'' = m\, f^{1-a} - c\, f$ and $T_{m,\ell}$ may be computed by
\begin{equation}\label{period}
  T_{m,\ell} = \sqrt{2}\int_{x_{m,\ell}^{-}}^{x_{m,\ell}^+}\frac{\dif x}{\sqrt{\ell-P_m(x)}},
\end{equation}
where $x_{m,\ell}^-$ and $x_{m,\ell}^+$ are the two positive roots of $P_m(x)=\ell$, with $x_{m,\ell}^-<x_{m,\ell}^+$.

Given \((\gamma_1, \gamma_2) \in \mathbb{R} \times \mathbb{R}^*\), let \(\Gamma_{T_{m,\ell}}(\gamma_1, \gamma_2) = (T_{m,\ell}, \gamma_1)\mathbb{Z} \oplus (0, \gamma_2)\mathbb{Z}\) be a lattice in \(\mathbb{R}^2\). We state the following result:

\begin{proposition}\label{prop:f}
Let $a, c \in \R^*$, with $a \,c > 0$. Then, there exists a two-parameter family of generalized Ricci metrics $\{\dif\sigma_{m,\ell}^2\}$ of type $(a, 0, c)$ on $\R^2$  that possess rotational symmetry, given by
\[
\dif\sigma^2_{m,\ell} = \dif s^2 + f_{m,\ell}(s)^2 \dif t^2,
\]
where $f_{m,\ell}$ is the positive $T_{m,\ell}$-periodic solution of equation $f'' = m\, f^{1-a} - c\, f$ with energy $\ell$. In addition, given \((\gamma_1, \gamma_2) \in \mathbb{R} \times \mathbb{R}^*\), $\R^2/\Gamma_{T_{m,\ell}}(\gamma_1, \gamma_2)$ endowed with the metric $\dif\sigma_{m,\ell}^2$ is a generalized Ricci torus of type $(a,0,c)$.
\end{proposition}
\begin{proof}
We observe that \(\dif\sigma_{m,\ell}^2\) defines a generalized Ricci metric of type \((a, 0, c)\) on \(\mathbb{R}^2\). Since \(f_{m,\ell}\) is a \(T_{m,\ell}\)-periodic function on \(\mathbb{R}\), the metric \(\dif\sigma_{m,\ell}^2\) descends to the torus \(\mathbb{R}^2 / \Gamma_{T_{m,\ell}}(\gamma_1, \gamma_2)\), for any \((\gamma_1, \gamma_2) \in \mathbb{R} \times \mathbb{R}^*\).
\end{proof}

\begin{remark}
The above result includes the class of Ricci metrics presented in \cite[Example 5.13]{Daniel-Zang}, referred to as Delaunay-type metrics, when $m = c$. Specifically, let $F_{m,\ell}(s)$ be defined as
\[
F_{m,\ell}(s) = \int_{0}^{s} \frac{\dif\varsigma}{f_{m,\ell}(\varsigma)}.
\]
Since $f_{m,\ell}(s)$ is a $T_{m,\ell}$-periodic function on $\R$, it is bounded and, consequently, $f_{m,\ell}(s)^{-1}$ is also bounded. Moreover, the function $F_{m,\ell}(s)$ satisfies
\begin{equation*}
F_{m,\ell}(s+n\, T_{m,\ell}) = F_{m,\ell}(s) + n\, F_{m,\ell}(T_{m,\ell}), \quad \text{for all } n \in \Z.
\end{equation*}
Thus, $F_{m,\ell}$ is an increasing diffeomorphism from $\R$ to $\R$. So $(s, t) \mapsto z = u(t) + i\, v(s)$, where $u(t) = t$ and $v(s) = F_{m,\ell}(s)$, defines a diffeomorphism between $\R^2$ and $\C$. Using the conformal parameter $z = u + i v$, the metric $\dif\sigma_{m,\ell}^2$ on $\C$ can be expressed as $e^{-2y}|\dif z|^2$, where $y:\C \to \R$ is a smooth function defined by $y(u + i\, v) = -\log f_{m,\ell}(s(v))$. The Gaussian curvature of $e^{-2y}|\dif z|^2$ is given by $K = 4 e^{2y}\,y_{z\bar{z}}$, and since $f_{m,\ell}$ satisfies $f'' = m\, f^{1-a} - c\, f$, we compute
\[
e^{-a\,y}(K - c) = -m.
\]
Moreover, a direct computation shows that $y(v)$ is a solution of
\[
\ddot{y} = c\, e^{-2y} - m\, e^{(a-2)y},
\]
where the dot denotes differentiation with respect to $v$. In particular, for $m= c$, we recover the Delaunay-type metrics.
\end{remark}

\section{The spherical Ricci tori}\label{sec:spherical-ricci}
From this section onward, we will focus on the case of spherical Ricci tori that possess rotational symmetry. Recall that, fixed $c>0$, a Riemannian surface \((\Sigma, \dif\sigma^2)\) is referred to as a \emph{$c$-spherical Ricci surface} if its Gaussian curvature $K$ satisfies
\begin{equation*}
(K - c)\,\Delta K - |\nabla K|^2 - 4K\,(K - c)^2 = 0.
\end{equation*}
In this context, the equation above is called the \emph{$c$-spherical Ricci condition}, and the metric \(\dif\sigma^2\) is said to be a \emph{$c$-spherical Ricci metric}.

Since every $c$-spherical Ricci surface is a generalized Ricci surface of type \((4,0,c)\) with \(c > 0\), it follows from Proposition~\ref{prop:l} that, for fixed \(c > 0\), there exists a two-parameter family \(\{\dif\sigma_{m,\ell}^2\}\) of $c$-spherical Ricci metrics, given by
\[
\dif\sigma^2_{m,\ell} = \dif s^2 + f_{m,\ell}(s)^2 \dif t^2,
\]
where \(f_{m,\ell}\) is the positive \(T_{m,\ell}\)-periodic solution of the equation \(f'' = m\, f^{-3} - c\, f\), with energy \(\ell\). Equivalently, \(f_{m,\ell}\) satisfies the first-order differential equation
\begin{equation}\label{EDO}
(f')^2 + c\, f^2 + m\, f^{-2} = 2\ell.
\end{equation}
In this setting, \(m\) must be positive since \(c > 0\), and \(\ell > \sqrt{c\,m}\) by Proposition~\ref{prop:l}, ensuring that \(\ell^2 - c\,m > 0\).

\subsection{An explicit family of solutions}
Fixed \(c > 0\), we define
\[
\Lambda_c = \{(m,\ell) \in \mathbb{R}^2 : m>0 \ \text{and} \ \ell > \sqrt{c\, m}\}.
\]

\begin{proposition}\label{prop:period}
Let \(c > 0\) and \((m,\ell) \in \Lambda_c\). If \(f\) is a periodic solution of equation  \eqref{EDO}, then its minimal period \(T_{m,\ell}\) is constant and equal to $\pi/\sqrt{c}$.
\end{proposition}
\begin{proof}
Since \(P_m(x) = (m\, x^{-2} + c\, x^2)/2\), introducing the variable \(u = c\, x^2 - \ell\), from \eqref{period} we can compute the minimal period \(T_{m,\ell}\)  as follows:
\[
T_{m,\ell} = 2 \int_{x_{m,\ell}^-}^{x_{m,\ell}^+} \frac{x\, \dif x}{\sqrt{2\ell\, x^2 - c x^4 - \ell}}
= \frac{1}{\sqrt{c}} \int_{-\sqrt{\ell^2 - c\,m}}^{\sqrt{\ell^2 - c\,m}} \frac{\dif u}{\sqrt{\ell^2 - c\,m - u^2}}
= \frac{\pi}{\sqrt{c}},
\]
where
\[
x_{m,\ell}^{\pm} = \sqrt{\frac{\ell \pm \sqrt{\ell^2 - c\,m}}{c}}.
\]
\end{proof}
In the sequel, for elements of \(\Lambda_c\) we will present an explicit family of periodic solutions of \eqref{EDO}.
\begin{lemma}\label{lemma:f}
Let \(c > 0\) and \((m, \ell) \in \Lambda_c\). Then
\[
f(s) = \sqrt{\frac{\ell + \sqrt{\ell^2 - c\, m}\, \sin(2\sqrt{c}\,s + c_1)}{c}}, \qquad c_1 \in \mathbb{R},
\]
is a positive, \(\tfrac{\pi}{\sqrt{c}}\)-periodic solution of equation\eqref{EDO}, well-defined on \(\mathbb{R}\). Furthermore, we have that
\[
0 < \frac{\ell - \sqrt{\ell^2 - c\, m}}{c} \leq f(s)^2 \leq \frac{\ell + \sqrt{\ell^2 - c\, m}}{c}, \quad \text{for all } s \in \mathbb{R}.
\]
\end{lemma}
\begin{proof}
We set \(g(s) = f(s)^2\). For \(c > 0\) and \((m, \ell) \in \Lambda\), by equation \eqref{EDO}, we obtain
\[
g'(s)^2 = -4c\, g^2 + 8\ell\, g - 4m = 4c\, \bigg[\frac{\ell^2 - c\, m}{c^2} - \left(g(s) - \frac{\ell}{c}\right)^2\bigg].
\]
By integrating the equation above, we obtain
\[
\arcsin\left(\frac{c\, g(s) - \ell}{\sqrt{\ell^2 - c\, m}}\right) = \pm 2\sqrt{c}\,s + c_1, \qquad c_1 \in \mathbb{R}.
\]
Since \(g(s) = f(s)^2\), and the sign in the sine argument can be absorbed into the constant \(c_1\), we find that
\[
f(s) = \sqrt{\frac{\ell + \sqrt{\ell^2 - c\, m}\, \sin(2\sqrt{c}\,s + c_1)}{c}}, \qquad c_1 \in \mathbb{R},
\]
solves the differential equation \eqref{EDO}. Thus, \(f(s)\) is a positive periodic function with period \(T = \pi/\sqrt{c}\), well-defined on \(\mathbb{R}\).
To establish bounds for \(f(s)^2\), we observe that it reaches its maximum value
$$\frac{\ell + \sqrt{\ell^2 - c\, m}}{c},\quad \text{when}\quad s = \frac{T}{4}+k\,T,$$ and its minimum value $$\frac{\ell - \sqrt{\ell^2 - c\, m}}{c},\quad \text{when}\quad s = \frac{3T}{4}+k\,T,$$ with $k\in\Z$. Moreover, the lower bound is positive since \(\sqrt{c\, m} < \ell\). This completes the proof.
\end{proof}

\begin{remark}
These explicit solutions depend continuously on the parameters \((m,\ell) \in \Lambda_c\). The constant solution is given by \(f = (m/c)^{1/4} = (\ell/c)^{1/2}\), corresponding to the degenerate case \(\ell^2 = c\,m\). Moreover, for \(m=0\) and \(\ell > 0\), the associated solution \(f\) is continuous but exhibits an infinite number of cusp points, making it non-smooth on \(\mathbb{R}\).
\end{remark}

To study the two-parameter family of metrics \(\{\dif\sigma_{m,\ell}^2\}\), we use the explicit solutions provided by Proposition~\ref{lemma:f} and we assume $c_1 = 0$. Fixing \(c > 0\), we consider the metric  given by
\[
\dif\sigma_{m,\ell}^2 = \dif s^2 + f_{m,\ell}(s)^2 \dif t^2, \qquad  (m, \ell) \in \Lambda_c,
\]
 and
\[
f_{m,\ell}(s) = \sqrt{\frac{\ell + \sqrt{\ell^2 - c\, m}\, \sin(2\sqrt{c}\,s)}{c}}.
\]

Although the metric \(\dif\sigma_{m,\ell}^2\) also depends on \(c\), we omit this dependence from the notation for simplicity. Denoting by \(K_{m,\ell}\) the Gaussian curvature of \(\dif\sigma_{m,\ell}^2\), we observe from Lemma~\ref{lemma:f} that
\begin{equation}\label{bound-of-K}
  K_{m,\ell}(s) - c = -m \, f_{m,\ell}(s)^4,
\end{equation}
which implies that \(K_{m,\ell}(s) < c\).

\begin{remark}
We observe two limiting cases. The curvature satisfies \(K_{m,\ell}(s) = c\) if and only if \(m = 0\), and \(K_{m,\ell}(s) = 0\) if and only if \(\ell^2 = c\, m\). In the latter case, \(f_{m,\ell}(s)\) is the constant solution given by \(f_{m,\ell}(s) \equiv (m/c)^{1/4}\).
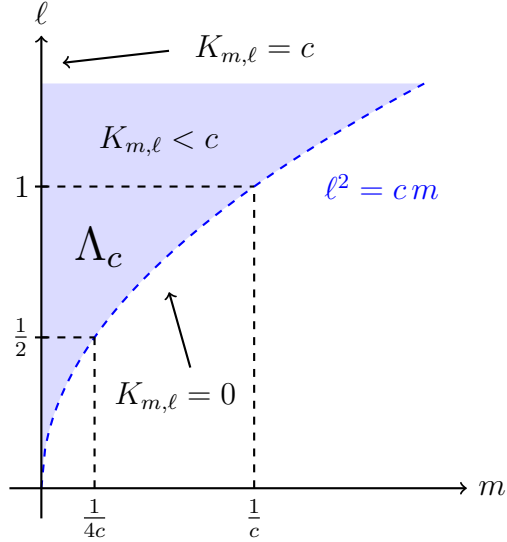
\begin{figure}[ht]
    \centering
\begin{tikzpicture}[yscale=4, xscale=2.8, thick]
 % Area tra la curva e la retta l = 1
  \fill[blue!20, opacity=0.7]
    (0,sqrt{1.8}) -- plot[domain=0:1.5, samples=100] (\x, {sqrt(\x)}) -- (1.8,sqrt{1.8}) -- cycle;

  % Curva l = sqrt(m)
  \draw[thick, blue, dashed, domain=0:1.8, samples=100]
    plot (\x, {sqrt(\x)});

  % Retta orizzontale l = 1
  \draw[thick,dashed] (0,1) -- (1,1) ;
  \draw[thick,dashed] (1,0) -- (1,1) ;
  \draw[thick,dashed] (1/4,0) -- (1/4,1/2) ;
  \draw[thick,dashed] (0,1/2) -- (1/4,1/2) ;

  % Assi
  \draw[->] (-0.15,0) -- (2,0) node[right] {$m$};
  \draw[->] (0,-0.1) -- (0,1.5) node[above] {$\ell$};

  % Etichette
  \node at (1,0) [below] {$\frac{1}{c}$};
    \node at (1/4,0) [below] {$\frac{1}{4c}$};
  \node at (0.55, 1.15) {$K_{m,\ell} < c$};
   \node at (1, 1.45) {$K_{m,\ell} = c$};
    \node at (0.63, 0.3) {$K_{m,\ell} = 0$};
    \node at (1.6, 1) {\textcolor{blue}{$\ell^2 = c\,m$}};
   \node at (0.27, 0.8) {\scalebox{1.5}{$\Lambda_c$}};;

   \foreach \y/\label in {0.5/{\tfrac{1}{2}}, 1/1} {
    \draw (-0.03,\y) -- (0.03,\y);  % Tacche
    \node[left] at (0,\y) {$\label$};  % Etichette
  }

  \draw[->, thick] (0.6,1.45) -- (0.1,1.41);
   \draw[->, thick] (0.7,0.4) -- (0.6,0.65);
\end{tikzpicture}
    \caption{The set $\Lambda_c$ of parameters.}
    \label{fig:lambda-set}
\end{figure}

\end{remark}

As a summary of this discussion, we conclude this section with one of the main results of this work. Given \((\gamma_1, \gamma_2) \in \mathbb{R} \times \mathbb{R}^*\), we establish the existence of a two-parameter family of $c$-spherical Ricci metrics that can be endowed on the torus \(\mathbb{R}^2 / \Gamma_{\frac{\pi}{\sqrt{c}}}(\gamma_1, \gamma_2)\), with
\[
\Gamma_{\frac{\pi}{\sqrt{c}}}(\gamma_1, \gamma_2) = \left(\frac{\pi}{\sqrt{c}}, \gamma_1\right)\mathbb{Z} \oplus (0 , \gamma_2)\mathbb{Z},
\]
where \( (\gamma_1, \gamma_2) \in \mathbb{R} \times \mathbb{R}^* \).

\begin{theorem}\label{thm:ricci-tori}
Let \(c > 0\) and \((\gamma_1, \gamma_2) \in \mathbb{R} \times \mathbb{R}^*\). The metrics \(\{\dif \sigma_{m,\ell}^2\}\), with \((m, \ell) \in \Lambda_c\), form a two-parameter family of rotationally invariant $c$-spherical Ricci metrics on the torus \(\mathbb{R}^2 / \Gamma_{\frac{\pi}{\sqrt{c}}}(\gamma_1, \gamma_2)\). Furthermore, there exist infinitely many metrics in this family that are not isometric.
\end{theorem}
\begin{proof}
The first part of the theorem follows directly from the preceding results and from the fact that equation~\eqref{EDO} is equivalent to the generalized Ricci condition of type \((4,0,c)\). Indeed, by Proposition~\ref{prop:l}, the level sets \(\{(x,y) : E_m(x,y) = \ell\}\) are compact smooth curves surrounding the equilibrium point. As a consequence, periodic solutions exist; their minimal periods are independent of both \(m\) and \(\ell\) (Proposition~\ref{prop:period}) and are explicitly described in Lemma~\ref{lemma:f}, by setting \(c_1 = 0\). Therefore, the two-parameter family \(\{\dif\sigma_{m,\ell}^2\}\) defines a set of metrics that can be endowed on the torus \(\mathbb{R}^2 / \Gamma_{\frac{\pi}{\sqrt{c}}}(\gamma_1, \gamma_2)\), for any \((\gamma_1, \gamma_2) \in \mathbb{R} \times \mathbb{R}^*\).

Additionally, from relation \eqref{bound-of-K}, we observe that the range of \((c - K_{m,\ell})\) is the closed interval \([L_1, L_2]\), where \(L_1\) and \(L_2\) are positive constants given by
\[
L_1 = \frac{m\, c^2}{(\ell+\sqrt{\ell^2 - c\, m})^2} \qquad \text{and} \qquad L_2 = \frac{m\, c^2}{(\ell - \sqrt{\ell^2 - c\, m})^2}.
\]
Let \((m, \ell), (m_*, \ell_*) \in \Lambda_c\) be such that \(\dif\sigma_{m, \ell}^2\) and \(\dif\sigma_{m_*, \ell_*}^2\) are isometric. By the Gauss Egregium Theorem, it follows that \(L_1 = L_1^*\) and \(L_2 = L_2^*\). Also, a straightforward computation yields \(\sqrt{m}\, \ell_* = \sqrt{m_*}\, \ell\).

Therefore, for fixed values \(m_0 > 0\) and \(\ell_0 > 0\), the families \(\{\dif \sigma_{m_0,\ell}^2\}\) and \(\{\dif \sigma_{m,\ell_0}^2\}\), with \((m_0, \ell), (m, \ell_0)\in \Lambda_c\), define one-parameter families of non-isometric metrics on the torus \(\mathbb{R}^2 / \Gamma_{\frac{\pi}{\sqrt{c}}}(\gamma_1, \gamma_2)\).

\end{proof}

\begin{corollary}
Every torus \(\mathbb{S}^1(r_1) \times \mathbb{S}^1(r_2)\) admits a \(\tilde{c}\)-spherical Ricci metric, for any \(\tilde{c} > 0\).
\end{corollary}
\begin{proof}
Set \(c = 1/4r_1^2\) and \((\gamma_1, \gamma_2) = (0, 2\pi r_2)\). Since \(\mathbb{S}^1(r_1) \times \mathbb{S}^1(r_2) \cong \mathbb{R}^2 / \Gamma_{2\pi r_1}(0, 2\pi r_2)\), Theorem~\ref{thm:ricci-tori} guarantees that \(\dif\sigma_{m, \ell}^2\) is a \(c\)-spherical Ricci metric on the torus for any \((m,\ell) \in \Lambda_c\). Consequently, the rescaled metric
\[
\frac{1}{2 r_1 \sqrt{\tilde{c}}} \, \dif\sigma_{m, \ell}^2
\]
is a \(\tilde{c}\)-spherical Ricci metric on \(\mathbb{S}^1(r_1) \times \mathbb{S}^1(r_2)\).
\end{proof}

\begin{remark}
    Let \((m,\ell)\in\Lambda_c\) be such that \(4c\, m = (1 - 2\ell)^2\). Setting \(B^2 = 4\ell - 1\) and \(H^2 = c\), we find that
    \(\dif\sigma_{m, \ell}^2\) can be written as
    \[
        \dif s^2 + \frac{\sqrt{1 + B^2 + 2B \sin(2Hs)}}{2|H|} \, \dif t^2,
    \]
    that is, the family \(\{\dif\sigma_{m, \ell}^2\}\) consists of the metrics induced by Delaunay surfaces in \(\mathbb{R}^3\) when \(4c\, m = (1 - 2\ell)^2\) (see, for example, \cite{Kenmotsu-2003}).
\end{remark}

\section{Spherical Ricci tori immersed in the \texorpdfstring{$3$}{3}-sphere}\label{sec:immersed-tori}

The goal of this section is to realize the spherical Ricci tori constructed in the previous section as rotationally invariant surfaces in the sphere \(\mathbb{S}^3\). Besides Lawson's results, which guarantee a minimal local isometric immersion of any Ricci metric of type \((4,0,c)\) into a simply connected three-dimensional space form of constant curvature \( c \), it is worth noting that, for \((c, m, 2\ell) = (1 + H^2,1/D^2, 1 - 2H/D)\), the equation \eqref{EDO}
\begin{equation*}
(f')^2 + c\, f^2 + m\, f^{-2} = 2\,\ell
\end{equation*}
plays a significant role in the construction of CMC-immersions  in \(\mathbb{S}^3\). This was first observed by Otsuki in \cite{Otsuki1970} for \( H = 0 \) and, later, extended by Wei in \cite{Wei2006} for \( H \neq 0 \). In particular, this equation establishes a one-to-one correspondence between CMC rotational surfaces in \(\mathbb{S}^3\) and positive solutions of the differential equation (we refer to \cite{Perdomo-2016} for a detailed exposition).

For this, we fix $c>0$ and let $\S^3_c$ be the $3$-dimensional sphere of sectional curvature $c$. For simplicity, we sometimes use the normalization $c = 1$. In this case, we denote $\S^3 = \S^3_1$. Also, we identify $\S^3_c$ with the sphere of radius $1/\sqrt{c}$ in $\R^4$, i.e.,
\[\S^3_c = \Big\{(x,y,z,w)\in\R^4 \ : \ x^2+y^2+z^2+w^2=\frac{1}{c}\Big\},\]
endowed with the induced metric from $\R^4$.
\subsection{Rotational surfaces}\label{rotational-surfaces} Given a geodesic $\gamma\subset\S^3_c$, let $\{\phi_t\}_{t\in\R}$ be the one-parameter subgroup of isometries of $\S^3_c$ that fixes $\gamma$. Without loss of generality, we may suppose that $\gamma$ is the great circle $\{(x,y,0,0)\in \S^3_c\}$, and so
\[\phi_t(x,y,z,w) = (x,y,z\cos t-w\sin t,z\sin t+w\cos t), \ \ t\in\S^1,\]
consists of rotations of angle $t$ around $\gamma$, whose the orbits are circles centered at $\gamma$. Thus, a surface in $\S^3_c$ is said {\it rotationally symmetric} with respect to the geodesic $\gamma$ in $\S^3_c$, also referred to as a {\it rotational surface}, if it is invariant under the action of the one-parameter subgroup of isometries $\{\phi_t\}$. A such surface may be locally parameterized by
\[X(s,t) = \phi_t(\alpha(s)),\]
where $\alpha:I\subset\R\to (\S^2_c)_+$ is a regular curve in the upper hemisphere  of a totally geodesic sphere $\S^2_c$ of $\S^3_c$ containing the geodesic $\gamma$. The curve $\alpha$ is called the {\it profile curve} of the parametrization $X$.

Without loss of generality, let $\alpha:I\subset\R\to (\S^2_c)_+$ be a regular curve parameterized by its arc length and defined on an open interval $I\subset\R$ containing $0$. We fix $(\S^2_c)_+ = \{(x,y,z,0)\in\S^3_c : z>0 \}$, and so $\alpha(s) = (x(s),y(s),z(s),0)$ where $z(s)$ is a positive function. Firstly, observe that $\alpha$ is entirely determined by $z(s)$, up to an integration. Indeed, we can write $x(s)$ and $y(s)$ as
$$\left\{\begin{aligned}
  x(s) &= \frac{1}{\sqrt{c}}\sqrt{1-c\, z(s)^2}\,\cos \theta(s), \\
  y(s) &= \frac{1}{\sqrt{c}}\sqrt{1-c\, z(s)^2}\,\sin \theta(s),
\end{aligned}
\right.
$$
for some function $\theta:I\to\R$. Also, since $\alpha$ is parameterized by its arc length we get that $\theta(s)$ satisfies the following ordinary differential equation:
\begin{equation*}
  \theta'(s)^2 = \frac{c-c^2\, z(s)^2-c\, z'(s)^2}{[1-c\, z(s)^2]^2}.
\end{equation*}
Up to a translation along the geodesic $\gamma$ and a change of orientation of $\alpha$, we may suppose that $\theta(0)=0$ and $\theta'(s)\geq 0$, for all $s$. Hence, $\alpha$ may be parameterized by
\begin{equation*}
  \alpha(s) = \left(\frac{1}{\sqrt{c}}\sqrt{1-c\, z(s)^2}\,\cos \theta(s),\frac{1}{\sqrt{c}}\sqrt{1-c\, z(s)^2}\,\sin \theta(s), z(s),0\right)
\end{equation*}
where $\theta(s)$ satisfies
\begin{equation}\label{theta-function}
  \theta'(s) = \sqrt{c}\,\frac{\sqrt{1-c\, z(s)^2-z'(s)^2}}{1-c \,z(s)^2} \quad \text{and} \quad \theta(0)=0.
\end{equation}

Therefore, up to congruences, any rotational surface $\Sigma$ in $\S_c^3$ can be locally parameterized by $X:I\times\S^1\to\S_c^3$ with
\begin{equation*}\label{parameterization}
  X(s,t)
   =\Big(\sqrt{c^{-1}- z(s)^2}\, \cos \theta(s),
   \sqrt{c^{-1}- z(s)^2} \,\sin \theta(s),
   z(s)\cos t,z(s)\sin t\Big)
\end{equation*}
for some smooth function $z(s):I\to \R$ satisfying $0< z(s)< \frac{1}{\sqrt{c}}$, for which the solution $\theta(s)$ of \eqref{theta-function} is well-defined on $I$. Moreover, using the parametrization \( X(s,t) \), we can show that the induced metric on \( \Sigma \) is given by
\begin{equation*}
  \dif \sigma^2 = \dif s^2 + z(s)^2 \dif t^2.
\end{equation*}
For an appropriate choice of orientation, we compute the mean curvature \( H \) of \( X(s,t) \) as
\begin{equation}\label{mean-curvature}
  H(s) = \frac{z(s)\,z''(s) + z'(s)^2 + 2c\, z(s)^2 - 1}{2 z(s)\, \sqrt{1 - c\, z(s)^2 - z'(s)^2}}.
\end{equation}

The solutions to the minimal equation for rotational surfaces in $\S^3_c$ comprise a one-parameter family $\{z_j(s)\}_{j\in[0,1]}$, where
\begin{equation}\label{eq:sol-minimal}
  z_j(s) = \sqrt{\frac{1+j\,\sin\left(2 \sqrt{c}\, s\right)}{2c}}.
\end{equation}
In addition, Ripoll's result in \cite[Theorem A]{Ripoll89} states that any minimal rotational surface in $\S^3_c$ is congruent to one generated by some $z_j(s)$, for $j\in[0,1].$
By equation \eqref{theta-function}, the function $\theta_j(s)$ can be expressed as
\[
\theta_j(s) = \int_{0}^{s} \frac{\sqrt{2\,(1-j^2)\,c}}{(1-j\,\sin(2\sqrt{c}\,\varsigma))\sqrt{1+j\,\sin(2\sqrt{c}\,\varsigma)}} \, \dif\varsigma,
\]
and, therefore, the profile curve associated with $z_j(s)$ is given by:
\[
\alpha_j(s) = \left(\frac{1}{\sqrt{c}}\sqrt{1-c\, z_j(s)^2}\,\cos \theta_j(s),\frac{1}{\sqrt{c}}\sqrt{1-c\,z_j(s)^2}\,\sin \theta_j(s), z_j(s), 0\right).
\]
Under these considerations, the study of minimal rotational surfaces in $\S^3_c$ can be divided into three distinct cases:
\begin{enumerate}
\item For $j = 0$, we get that $z_{0}(s) = 1/\sqrt{2c}$ and $\theta_{0}(s) = \sqrt{2c}\, s$, whose are defined on $\R$. In this case, the profile curve $\alpha_{0}(s)$ is closed for $s\in[0,2\pi/\sqrt{2c}]$ and the map $X_{0}:\S^1(1/\sqrt{2c})\times\S^1\to\S^3_c$
parameterizes the Clifford torus $\S^1(1/\sqrt{2c})\times\S^1(1/\sqrt{2c})$.
\item For $j=1$, we get that $\theta_1(s)=0$ and
    \[z_1(s) =\frac{\sin\left(\sqrt{c}\,s+\frac{\pi}{4}\right)}{\sqrt{c}},\] with $I = (-\frac{\pi}{4\sqrt{c}},\frac{3\pi}{4\sqrt{c}})$. The profile curve $\alpha_1(s)$ parameterizes a piece of a great circle of $\S^3_c$. By standard arguments $\dif s^2 + z_1(s)^2\dif t^2$ extends to a smooth metric on $\Omega=[-\frac{\pi}{4\sqrt{c}},\frac{3\pi}{4\sqrt{c}}]\times\S^1$
    and we get that $X_1:\Omega\to\S^3_c$ parametrizes the totally geodesic sphere $ \{(x,0,z,w)\in\S^3_c\}$.

\item For $j \in(0,1)$, we have that $z_j(s)$ is a $\tfrac{\pi}{\sqrt{c}}$-periodic function. The function $\theta_j(s)$ is not periodic and the profile curve $\alpha_j(s)$ is a closed curve if, and only if, $m\, \theta_j\big(\tfrac{\pi}{\sqrt{c}}\big)=2n\pi$, for all $m,n\in\Z$, or equivalently when
    \[T(j) =\frac{\theta_j\left(\frac{\pi}{\sqrt{c}}\right)}{2\pi}\in\Q.\]
    In this context, these surfaces can be generated either by closed curves or by curves that are not closed, and this problem is analyzed by studying the period function $T(j)$. In both scenarios, $X_j(s,t)$ parameterizes \emph{spherical catenoids}, since their profile curves are geometrically classified as \emph{spherical catenaries}. Furthermore, the closed spherical catenaries are non-embedded, possess dihedral symmetry, and give rise to the Otsuki minimal tori in $\S^3_c$ (cf. \cite[Corollary 3.
    5 and Theorem 5.2]{Castro24} for more details).
\end{enumerate}

\subsection{Immersing spherical Ricci tori into \texorpdfstring{$\mathbb{S}^3_c$}{S3c}}
To realize $c$-spherical Ricci tori as complete rotationally invariant surfaces in \(\mathbb{S}^3_c\), using the two-parameter family of metrics presented in Theorem~\ref{thm:ricci-tori}, we first note, by Lemma~\ref{lemma:f}, that \( 0 < f_{m,\ell}(s)^2 < 1/c \) if and only if \(\ell < 1\) and \(\ell < (1 + c\, m)/2\). Thus, we define the following set of admissible pairs:
\[
\Lambda_c' = \left\{(m,\ell) \in \mathbb{R}^2 : 0 < m < \frac{1}{c} \quad \text{and} \quad \sqrt{c\,m} < \ell < \frac{c\,m + 1}{2} \right\} \subset \Lambda_c.
\]
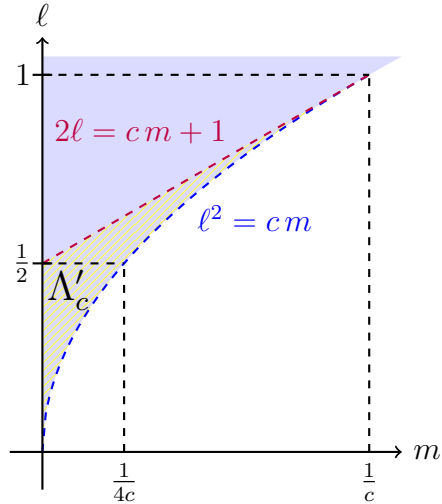
\begin{figure}[ht]
    \centering
    \begin{tikzpicture}[yscale=5, xscale=4.3, thick]

        % Área entre l = (m+1)/2 e l = 1 (azul claro)
        \fill[blue!20, opacity=0.7]
            (0,sqrt{1.1}) -- plot[domain=0:1.1, samples=100]
            (\x, {sqrt(\x)}) --
            (1.1,sqrt{1.1}) -- cycle;

        % Nova área: entre sqrt(m) e (m+1)/2, com rachura azul
        \fill[pattern=north east lines, pattern color=yellow!70, domain=0:1, variable=\x]
            plot[domain=0:1, samples=100] ({\x}, {sqrt(\x)})
            -- plot[domain=1:0, samples=100] ({\x}, {(\x+1)/2})
            -- cycle;

        % Curva l = sqrt(m)
        \draw[thick, blue, dashed, domain=0:1, samples=100]
            plot (\x, {sqrt(\x)});

        % Curva l = (m+1)/2
        \draw[thick, purple, dashed, domain=0:1, samples=100]
            plot (\x, {(\x+1)/2});

        % Retta orizzontale l = 1
        \draw[thick,dashed] (0,1) -- (1,1) ;
        \draw[thick,dashed] (1,0) -- (1,1) ;
        \draw[thick,dashed] (1/4,0) -- (1/4,1/2) ;
        \draw[thick,dashed] (0,1/2) -- (1/4,1/2) ;

        % Assi
        \draw[->] (-0.1,0) -- (1.1,0) node[right] {$m$};
        \draw[->] (0,-0.1) -- (0,1.1) node[above] {$\ell$};

        % Etichette
        \node at (1,0) [below] {$\frac{1}{c}$};
        \node at (1/4,0) [below] {$\frac{1}{4c}$};
        \node at (0.08, 0.43) {\scalebox{1.3}{$\Lambda_c'$}};
        \node at (0.65, 0.62) {\textcolor{blue}{$\ell^2 = c\,m$}};
        \node at (0.3, 0.85) {\textcolor{purple}{$2\ell = c\,m+1$}};

        \foreach \y/\label in {0.5/{\tfrac{1}{2}}, 1/1} {
            \draw (-0.03,\y) -- (0.03,\y);  % Tacche
            \node[left] at (0,\y) {$\label$};  % Etichette
        }
    \end{tikzpicture}
    \caption{The admissible set $\Lambda'_c$.}
    \label{fig:lambda'-set}
\end{figure}

\begin{remark}
For every \(c > 0\), we have \(\{(m, \tfrac{1}{2}) : 0 < m < \tfrac{1}{4c}\} \subset \Lambda_c'\). These parameter choices correspond to the induced metrics of minimal rotational surfaces in \(\mathbb{S}^3_c\), since, by \eqref{eq:sol-minimal}, the function \(z_j(s)\), with \(j \in (0,1)\), coincides with \(f_{m,\tfrac{1}{2}}(s)\), for \(m \in (0,\tfrac{1}{4c})\).
\end{remark}

In the following, we will show that for any \( (m,\ell) \in \Lambda_c' \), the function \(\theta_{m,\ell}(s)\), associated with \(f_{m,\ell}(s)\), is an increasing diffeomorphism defined on \(\mathbb{R}\).

\begin{proposition}\label{prop:theta}
  Let \(c > 0\) and \((m, \ell) \in \Lambda_c'\). Then, the function
\[
\theta_{m,\ell}(s) = \sqrt{c}\, \int_{0}^s \frac{\sqrt{m+(1 - 2\ell)\,f_{m,\ell}^2(\varsigma)}}{f_{m,\ell}(\varsigma)\,(1 - c\, f_{m,\ell}^2(\varsigma))}\, \dif\varsigma
\]
is a well-defined increasing diffeomorphism on \(\mathbb{R}\), satisfying
\[
\theta_{m,\ell}\left(s + n\,\frac{\pi}{\sqrt{c}}\right) = \theta_{m,\ell}(s) + n\, \theta_{m,\ell}\left(\frac{\pi}{\sqrt{c}}\right),
\]
for all \(s \in \mathbb{R}\) and \(n \in \mathbb{Z}\).
\end{proposition}
\begin{proof}
Consider \((m,\ell) \in \Lambda_c'\). Since
\[
1 - c\, f_{m,\ell}(s)^2 - f'_{m,\ell}(s)^2 = \frac{m}{f_{m,\ell}(s)^2} + 1 - 2\ell,
\]
Lemma~\ref{lemma:f} ensures the existence of real constants \(L_1\) and \(L_2\) such that
\[
L_1 \leq 1 - c\, f_{m,\ell}(s)^2 - f'_{m,\ell}(s)^2 \leq L_2.
\]
Moreover, a direct computation yields
\[
L_1 = 1 - \ell - \sqrt{\ell^2 - c\, m} = 1 - c\, \max f_{m,\ell}(s)^2 > 0,
\]
since \((m,\ell)\in\Lambda_c'\). Consequently, for each \((m,\ell) \in \Lambda_c'\), there exist positive constants \(M_1\) and \(M_2\) such that
\[
M_1 \leq \sqrt{c}\,\frac{\sqrt{m+(1 - 2\ell)\,f_{m,\ell}^2(s)}}{f_{m,\ell}(s)\,(1 - c\, f_{m,\ell}^2(s))} \leq M_2.
\]
By the Cauchy–Lipschitz theorem, there exists a unique smooth function \(\theta_{m,\ell}(s)\), well-defined on \(\mathbb{R}\), satisfying
\begin{equation}\label{edo-theta}
\begin{cases}
  \theta'(s) = \displaystyle\sqrt{c}\,\frac{\sqrt{m+(1 - 2\ell)\,f_{m,\ell}(s)^2}}{f_{m,\ell}(s)\,(1 - c\,f_{m,\ell}(s)^2)}, \qquad
  \theta(0) = 0,
\end{cases}
\end{equation}
for each \((m,\ell) \in \Lambda_c'\). Furthermore, since \(\theta'_{m,\ell}(s)\) is bounded below by \(M_1 > 0\), it follows that \(\theta_{m,\ell}(s) \to \pm\infty\) as \(s \to \pm \infty\), so that \(\theta_{m,\ell}\) is an increasing diffeomorphism from \(\mathbb{R}\) onto \(\mathbb{R}\).
Now consider the function $$\tilde{\theta}_{m,\ell}(s) = \theta_{m,\ell}\Big(s + \frac{\pi}{\sqrt{c}}\Big) - \theta_{m,\ell}\Big(\frac{\pi}{\sqrt{c}}\Big).$$ Since \(f_{m,\ell}(s)\) is \(\tfrac{\pi}{\sqrt{c}}\)-periodic, it follows that \(\tilde{\theta}_{m,\ell}(s)\) also satisfies \eqref{edo-theta} on \(\mathbb{R}\). By uniqueness, we must have \(\tilde{\theta}_{m,\ell}(s) = \theta_{m,\ell}(s)\). Hence, for all \(n \in \mathbb{Z}\),
\[
\theta_{m,\ell}\left(s + n \,\frac{\pi}{\sqrt{c}}\right) = \theta_{m,\ell}(s) + n\, \theta_{m,\ell}\left(\frac{\pi}{\sqrt{c}}\right).
\]
\end{proof}
For \((m,\ell) \in \Lambda_c'\), consider the regular curve \(\alpha_{m,\ell}: \mathbb{R} \to (\mathbb{S}^2_c)_+\), defined by
\[
\alpha_{m,\ell}(s) = \left( \frac{1}{\sqrt{c}} \sqrt{1 - c\, f_{m,\ell}(s)^2}\, \cos \theta_{m,\ell}(s), \frac{1}{\sqrt{c}} \sqrt{1 - c\, f_{m,\ell}(s)^2} \,\sin \theta_{m,\ell}(s),f_{m,\ell}(s),\; 0 \right).
\]
We observe that \(\alpha_{m,\ell}\) is a closed curve in \(\mathbb{S}^3_c\) if and only if \(\theta_{m,\ell}(\tfrac{\pi}{\sqrt{c}}) / 2\pi \in \mathbb{Q}\), and it is embedded when \(\theta_{m,\ell}(\tfrac{\pi}{\sqrt{c}}) = 2\pi/n\), for some positive integer \(n\). As a consequence of the above discussion, we obtain the following theorem:

\begin{theorem}\label{thm:immersions}
  Let \(c > 0\) and \((m,\ell) \in \Lambda_c'\), and consider the map \(X_{m,\ell}: \mathbb{R}^2 \to \mathbb{S}^3_c\) defined by \(X_{m,\ell}(s,t) = \phi_t(\alpha_{m,\ell}(s))\).
  Then, \(X_{m,\ell}\) defines an isometric immersion of the $c$-spherical Ricci surface \((\mathbb{R}^2, \dif \sigma_{m,\ell}^2)\) into \(\mathbb{S}^3_c\). Moreover, if \(\theta_{m,\ell}(\tfrac{\pi}{\sqrt{c}})/2\pi \in \mathbb{Q}\), then \(X_{m,\ell}(\mathbb{R}^2)\) is a compact surface in \(\mathbb{S}^3_c\), and it is embedded when \(\theta_{m,\ell}(\tfrac{\pi}{\sqrt{c}}) = 2\pi/n\), for some positive integer \(n\).
\end{theorem}

\begin{remark}
 The immersion \(X_{m,\ell}\) is minimal if and only if \(0 < m < 1/4c\) and \(\ell = 1/2\).
\begin{figure}[ht]
    \centering
    \begin{tikzpicture}[yscale=5, xscale=5, thick]

        % Área entre l = (m+1)/2 e l = 1 (azul claro)
        \fill[blue!20, opacity=0.7]
            (0,sqrt{1.4}) -- plot[domain=0:1.4, samples=100]
            (\x, {sqrt(\x)}) --
            (1.1,sqrt{1.4}) -- cycle;

        % Nova área: entre sqrt(m) e (m+1)/2, com rachura azul
        \fill[pattern=north east lines, pattern color=yellow!70, domain=0:1, variable=\x]
            plot[domain=0:1, samples=100] ({\x}, {sqrt(\x)})
            -- plot[domain=1:0, samples=100] ({\x}, {(\x+1)/2})
            -- cycle;

        % Curva l = sqrt(m)
        \draw[thick, blue, dashed, domain=0:1.4, samples=100]
            plot (\x, {sqrt(\x)});

        % Curva l = (m+1)/2
        \draw[thick, purple, dashed, domain=0:1, samples=100]
            plot (\x, {(\x+1)/2});

        % Curva m = (1/2-l)^2
        \draw[thick, olive, domain=0.25:1.18, samples=100, variable=\y]
        plot ({(1/2 - \y)^2}, \y);

        % Retta orizzontale l = 1
        \draw[thick,dashed] (0,1) -- (1,1) ;
        \draw[thick,dashed] (1,0) -- (1,1) ;
        \draw[thick,dashed] (1/4,0) -- (1/4,1/2);
        \draw[thick,teal] (0,1/2) -- (1/4,1/2) ;

        % legendas
        \draw[decorate, decoration={brace}, thick]
    (0,0.55) -- (0.25,0.55)
    node[midway, yshift=11pt] {\footnotesize (I)};
        \node at (0.33,0.47) {\footnotesize (II)};
        \node at (0.53,1.1) {\footnotesize (III)};
        %\node[above] at (0.125, 0.5) {\tiny $\ell = \frac{1}{2}$};
        \node[rotate=90] at (-0.03,0.76) {\tiny Round metrics};
        \node at (0.8,0.7) {\tiny Flat metrics};
        \draw[->, thick] (0.63,0.7) -- (0.57,0.73);

        % ponto
        \fill[gray] (0.25, 0.5) circle (0.5pt);

        % Assi
        \draw[->] (-0.1,0) -- (1.4,0) node[right] {$m$};
        \draw[->] (0,-0.1) -- (0,1.2) node[above] {$\ell$};

        % Etichette
        \node at (1,0) [below] {\tiny $1$};
        \node at (1/4,0) [below] {$\frac{1}{4}$};
        \node at (0.08, 0.43) {\scalebox{1.3}{$\Lambda_1'$}};
        %\node at (0.65, 0.62) {\textcolor{blue}{$\ell^2 = c\,m$}};
        %\node at (0.3, 0.85) {\textcolor{purple}{$2\ell = c\,m+1$}};

        \foreach \y/\label in {0.5/{\tfrac{1}{2}}, 1/$\tiny{1}$} {
            \draw (-0.03,\y) -- (0.03,\y);  % Tacche
            \node[left] at (0,\y) {$\label$};  % Etichette
        }
    \end{tikzpicture}
    \caption{The admissible set \(\Lambda'_1\): (I) metrics induced on minimal rotational surfaces in \(\mathbb{S}^3\); (II) the induced metric of the Clifford torus; (III) metrics induced on Delaunay surfaces in \(\mathbb{R}^3\).}
    \label{fig:lambda'-set2}
\end{figure}
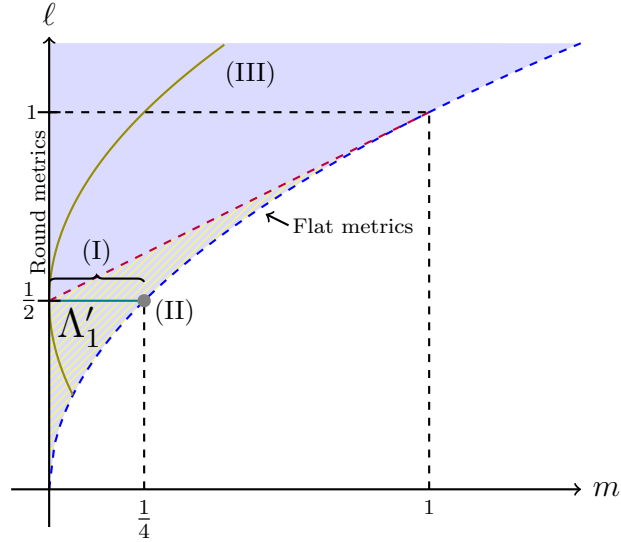

\end{remark}
\subsection{Some explicit examples} Fix $c=1$. To construct compact surfaces in \( \mathbb{S}^3 \) by means of Theorem~\ref{thm:immersions}, it suffices to find parameters \( (m, \ell) \in \Lambda_1' \) such that
\[
\frac{\theta_{m,\ell}(\pi)}{2\pi} \in \mathbb{Q}.
\]

We consider the function \(\Theta:\Lambda_1'\to \mathbb{R}\) defined by \(\Theta(m,\ell) = \theta_{m,\ell}(\pi)\). Observe that \(\Theta\) is continuous and its integrand is continuous and well-defined for all \((m,\ell)\in\Lambda'_1\) and \(s\in\mathbb{R}\).
Moreover, by means of an explicit computation of the bounds of $\Theta(m,\ell)$ using Lemma~\ref{lemma:f}, one verifies that
\[
\lim_{m\to\ell^2}\Theta(m,\ell)
=\frac{\pi}{\sqrt{1-\ell}}, \quad \text{for each } \ell\in(0,1),
\]
and
\[
\lim_{m\to 0}\Theta(m,\ell)
=\pi, \quad \text{for each } \ell\in\left(0,\tfrac{1}{2}\right].
\]

\begin{remark}\label{rmk:embedded-examples}
In particular, no embedded examples can occur for $\ell \leq \tfrac{1}{2}$. In fact, for \((m,\ell)\in\Lambda'_1\) with $\ell \in \left(0,\tfrac{1}{2}\right]$, the condition \(\theta_{m,\ell}(\pi) = \tfrac{2\pi}{n}\) holds if and only if
\[
\pi < \frac{2\pi}{n} < \frac{\pi}{\sqrt{1-\ell}}.
\]
This inequality forces $n=1$ and $4 < \tfrac{1}{1-\ell}$, which is impossible. In the distinguished case $\ell = \tfrac{1}{2}$, the function $m\mapsto \Theta(m,\tfrac{1}{2})$ is the period function of rotational minimal surfaces in $\S^3$, which is monotonically increasing and maps $(0,\tfrac{1}{4})$ onto $(\pi,\sqrt{2}\pi)$ (cf.~\cite{Ben-Li-2015,Otsuki1970,Perdomo-2016}).
\end{remark}

We next show that for $\ell>1/2$ there exist embedded compact examples.

\begin{theorem} \label{thm:embedded}
For each \(m\in(0,\tfrac{9}{16})\), there exists \(\ell\in(\tfrac{1}{2},\tfrac{25}{32})\) such that \((m,\ell)\in\Lambda'_1\) and \(\Theta(m,\ell) = 2\pi\). Consequently, for such parameters \((m,\ell)\), the immersion \(X_{m,\ell}\) defines an embedded compact spherical Ricci surface in \(\mathbb{S}^3\).
\end{theorem}

\begin{proof}
Fix \(m\in(0,1)\). From Lemma~\ref{lemma:f}, after computing explicitly the bounds of \(\Theta(m,\ell)\), we find
\[
\lim_{\ell\to\sqrt{m}}\Theta(m,\ell) = \frac{\pi}{\sqrt{1-\sqrt{m}}}.
\]
Hence, for \(0<m<\tfrac{9}{16}\), this limit satisfies \(\displaystyle\lim_{\ell\to\sqrt{m}}\Theta(m,\ell)<2\pi\).
On the other hand, for a fixed \(m\in(0,1)\), it results that
\[
\lim_{\ell\to \tfrac{m+1}{2}}\Theta(m,\ell)
= \int_{0}^{\tfrac{\pi}{4}} g(m,\varsigma)\,\dif\varsigma
  + \int_{\tfrac{\pi}{4}}^{\pi} g(m,\varsigma)\,\dif\varsigma,
\]
where
\[
g(m,\varsigma ) = \frac{2\sqrt{m}}{\sqrt{1-m}\,\sqrt{1-\sin(2\varsigma )}\,\sqrt{1+m+(1-m)\sin(2\varsigma )}}.
\]
Although both integrals are improper, they can be evaluated explicitly, since
\[
\int g(m,\varsigma)\,\dif\varsigma
= \frac{\sqrt{m}}{\sqrt{1-m}}
\tanh^{-1}\!\left(
\frac{\cos(2\varsigma)}{\sqrt{1-\sin(2\varsigma )}\,\sqrt{1+m+(1-m)\sin(2\varsigma )}}
\right).
\]
Therefore,
\begin{align*}
  \lim_{\ell\to \tfrac{m+1}{2}}\Theta(m,\ell)  &= \lim_{\varepsilon\to 0^+}\left(\int_{0}^{\tfrac{\pi}{4}-\varepsilon} g(m,\varsigma)\,\dif\varsigma
  + \int_{\tfrac{\pi}{4}+\varepsilon}^{\pi} g(m,\varsigma)\,\dif\varsigma\right)\\
  &= \lim_{\varepsilon\to 0^+}
\frac{2\sqrt{m}}{\sqrt{1-m}}
\tanh^{-1}\!\left(
\frac{\sqrt{2}\,\cos \varepsilon}{\sqrt{1+m+(1-m)\cos(2\varepsilon)}}
\right)
= +\infty.
\end{align*}

By continuity, it follows that for each \(m\in(0,\tfrac{9}{16})\) there exists \(\ell\in(\sqrt{m},\tfrac{m+1}{2})\) such that \(\Theta(m,\ell)=2\pi\). Finally, Remark~\ref{rmk:embedded-examples} ensures that  \(\ell\in(\tfrac{1}{2},\tfrac{25}{32})\), completing the proof.
\end{proof}

We are able to construct some explicit examples of $1$-spherical Ricci tori by analyzing $\Theta(m,\ell)$ numerically. In the following, we present one immersed and one embedded example, illustrated via the stereographic projection of $\mathbb{S}^3$ onto $\mathbb{R}^3$.

\begin{example} Let $m = 0.51$. For $\ell = 0.73$, we find that $\Theta(m,\ell) = 2\pi$. In this case, the generating curve is a closed embedded curve lying in the upper hemisphere of a totally geodesic $2$-sphere $\mathbb{S}^2 \subset \mathbb{S}^3$.

\begin{figure}[!ht]
    \begin{subfigure}
        \centering
        \includegraphics[width=0.35\linewidth]{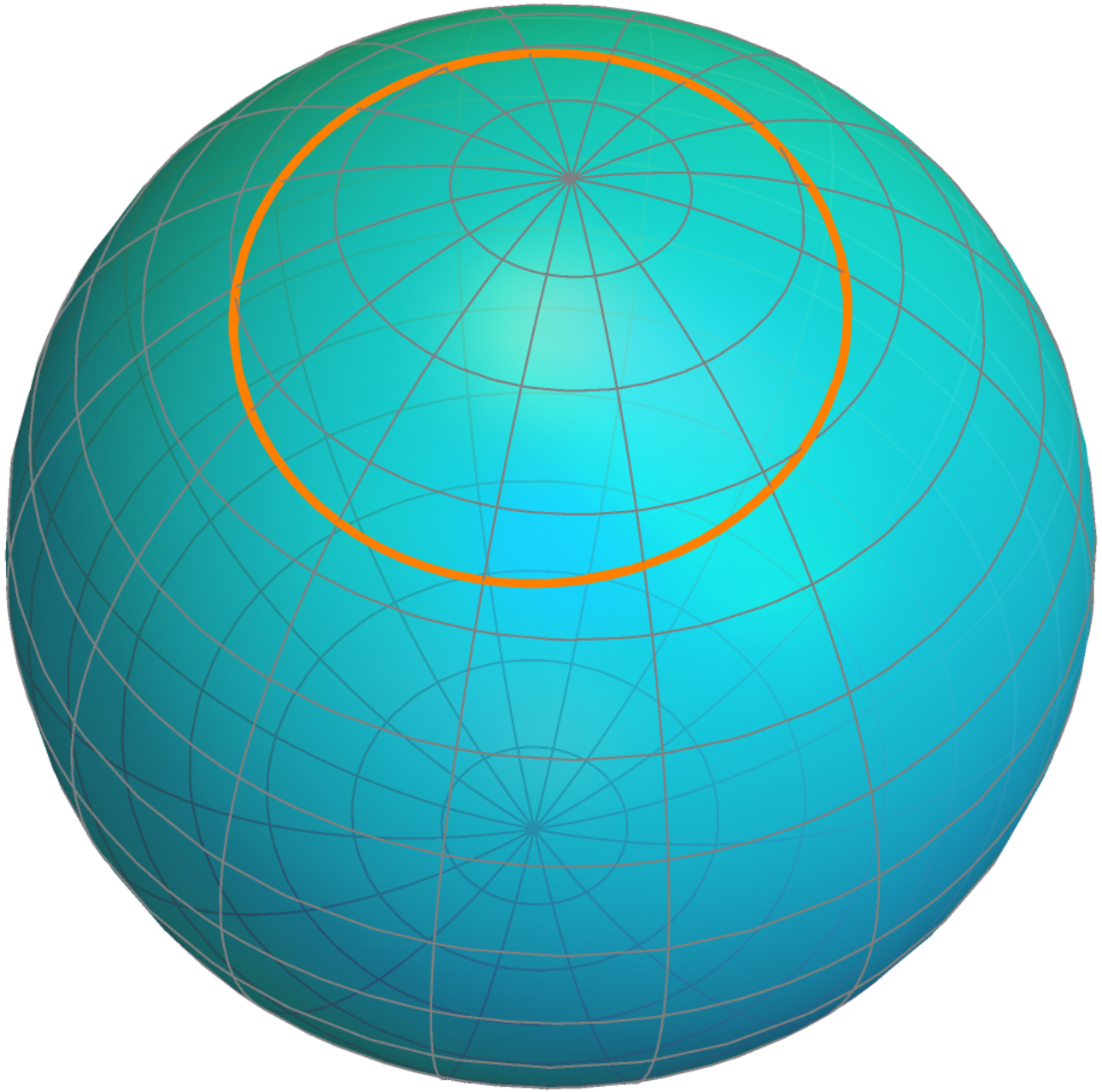}
    \end{subfigure}\hspace{1.5cm}
    \begin{subfigure}
        \centering
        \includegraphics[width=0.45\linewidth]{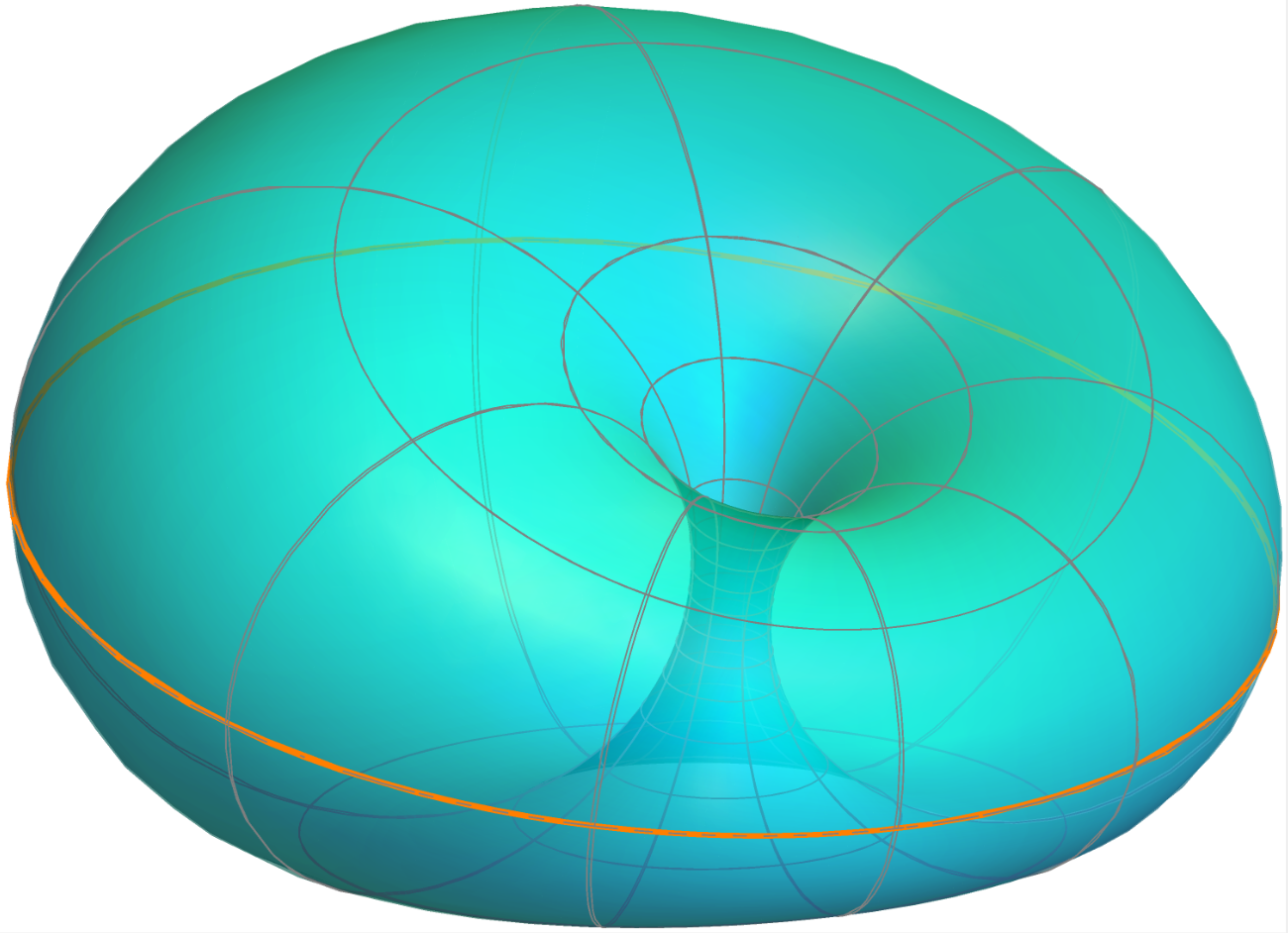}
    \end{subfigure}
    \caption{Generating curve and embedded $1$-spherical Ricci torus with $m = 0.51$ and $\ell = 0.73$.}
\end{figure}
\end{example}

\begin{example}  Let $m = 0.75$. For $\ell = 0.8700024$, we find that $\Theta(m,\ell) = 3\pi$. In this case, the generating curve is a closed immersed curve lying in the upper hemisphere of a totally geodesic $2$-sphere $\mathbb{S}^2 \subset \mathbb{S}^3$.

\begin{figure}[!ht]
    \begin{subfigure}
        \centering
        \includegraphics[width=0.35\linewidth]{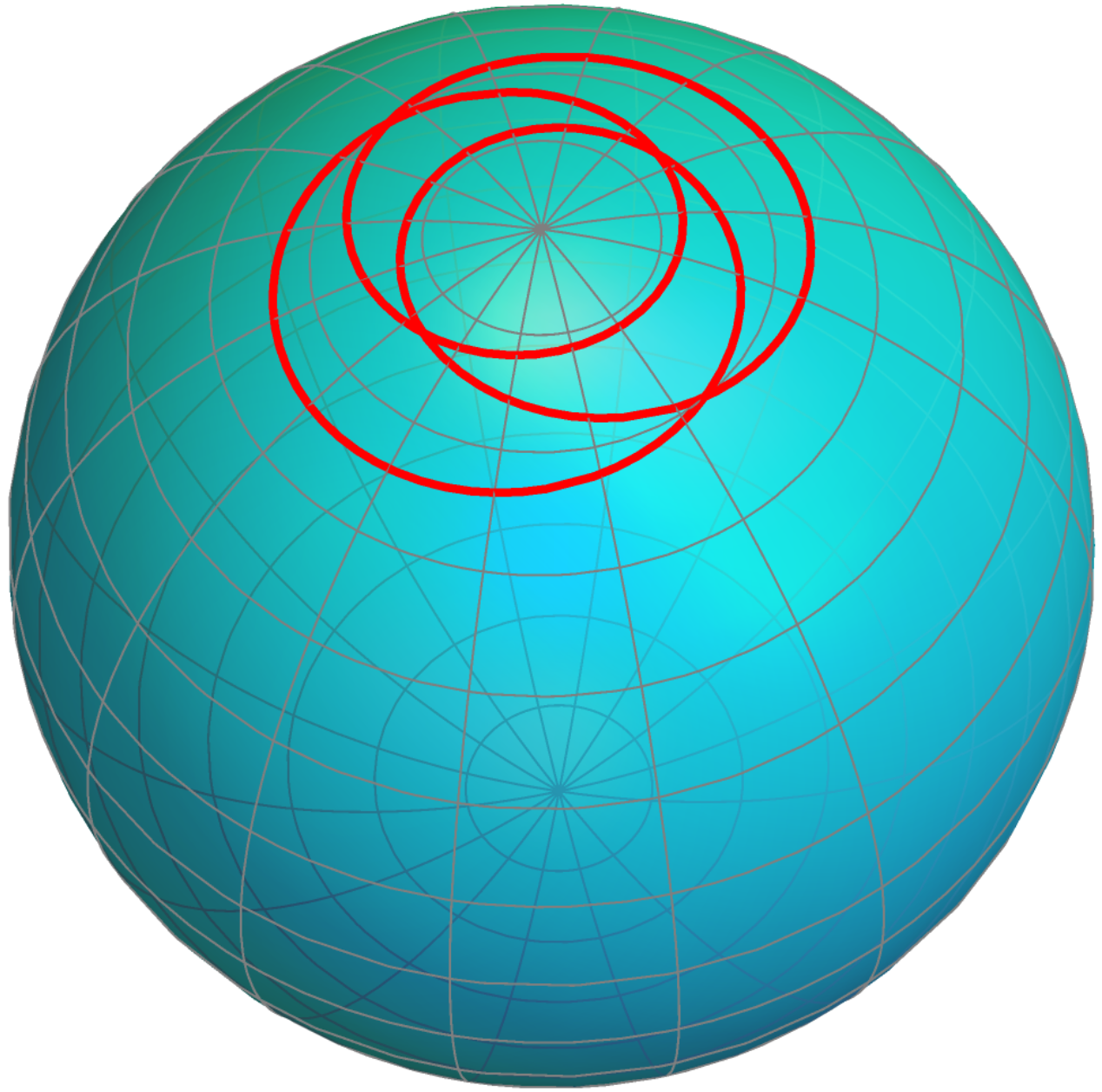}
    \end{subfigure}\hspace{1.5cm}
    \begin{subfigure}
        \centering
        \includegraphics[width=0.46\linewidth]{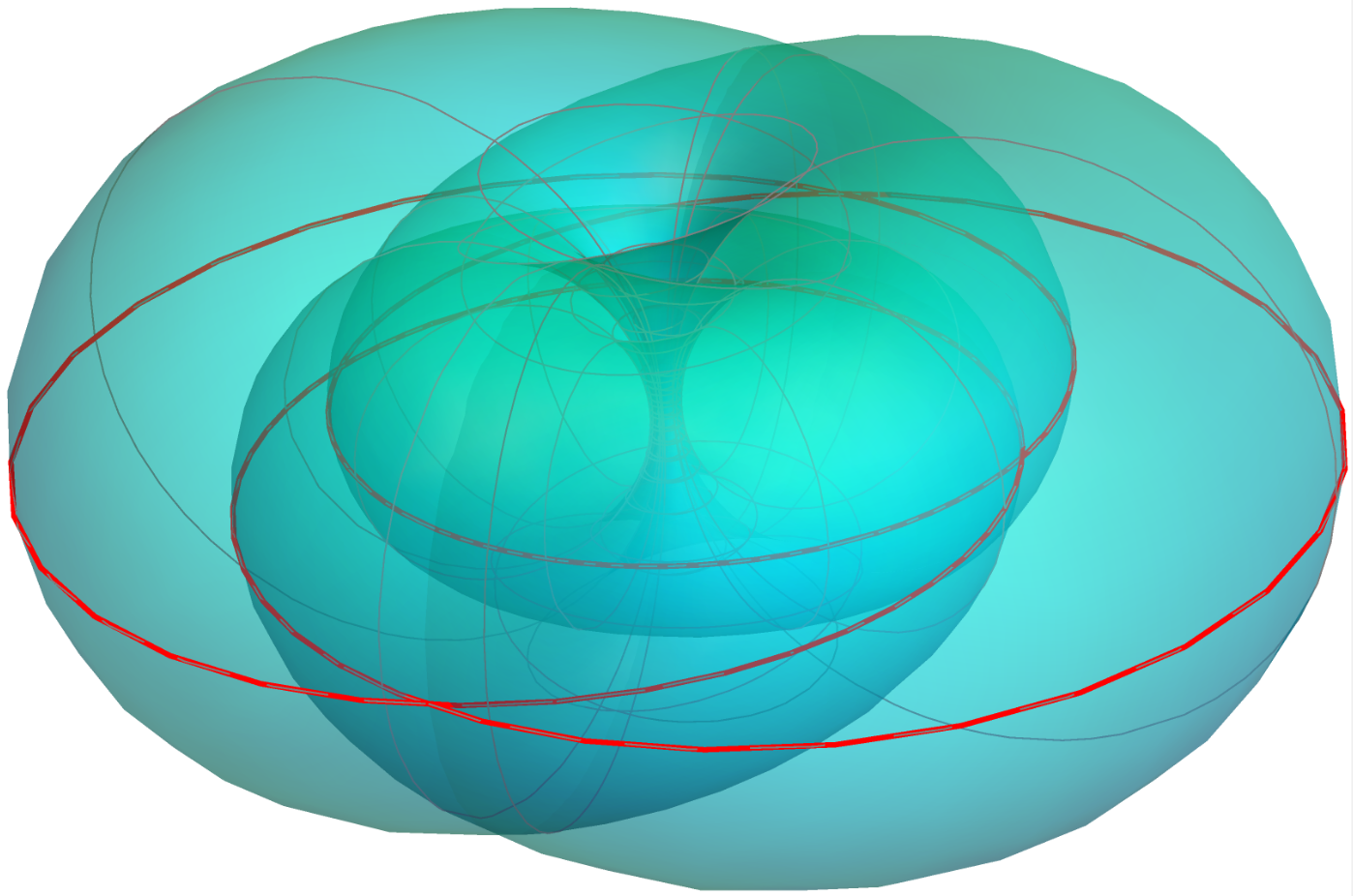}
    \end{subfigure}
    \caption{Generating curve and immersed $1$-spherical Ricci torus with $m = 0.75$ and $\ell \approx 0.87$.}
\end{figure}
\end{example}
\newpage

\bibliographystyle{amsplain}
		\bibliography{references}
\end{document}